\documentclass[12pt]{article}

\usepackage[english]{babel}
\usepackage{upref,amsfonts,amsxtra,latexsym,color}
\usepackage{amssymb,amsmath,amsthm,epsf,a4wide,mathrsfs,verbatim,hyperref}
\usepackage[all]{xy}

\newtheorem{theorem}{Theorem}[subsection]
\newtheorem{lemma}[theorem]{Lemma}
\newtheorem{corollary}[theorem]{Corollary}
\newtheorem{proposition}[theorem]{Proposition}

\theoremstyle{definition}
\newtheorem{definition}[theorem]{Definition}
\newtheorem{remark}[theorem]{Remark}
\newtheorem{example}[theorem]{Example}
\newtheorem{question}[theorem]{Question}

\numberwithin{equation}{section}
\numberwithin{theorem}{section}

\newcommand{\cM}{{\cal M}}
\newcommand{\ep}{{\varepsilon}}
 
\newcommand{\cA}{{\cal A}}

\newcommand{\cP}{{\cal P}}

\newcommand{\cB}{{\cal B}}

\newcommand{\cO}{{\cal O}}

\newcommand{\cI}{{\cal I}}
\newcommand{\cR}{{\cal R}}
\newcommand{\cZ}{{\mathcal Z}}
\newcommand{\C}{{\mathbb C}}
\newcommand{\Z}{{\mathbb Z}}
\newcommand{\N}{{\mathbb N}}

\newcommand{\T}{{\mathbb T}}

\newcommand{\cK}{{\cal K}}

\newcommand{\R}{{\mathbb R}}
\newcommand{\Cs}{{$C^*$-al\-ge\-bra}}

\newcommand{\sh}{{$^*$-ho\-mo\-mor\-phism}}

\date{}

\title{Around traces and quasitraces}

\author{Henning Olai Milh\o j and Mikael R\o rdam{\thanks{The second named author was supported by a research grant from the Danish Council for Independent Research, Natural Sciences.}}}

\begin{document}

\maketitle

\begin{center}
\emph{Dedicated to the memory of Eberhard Kirchberg}
\end{center}

\begin{abstract} 
\noindent This paper presents a survey of results on traces and quasitraces on \Cs s, and it provides some new results  on traces on ultraproducts and on the existence of faithful traces. As for the former, we exhibit a sequence of traceless simple, separable, unital, nuclear \Cs s whose ultraproduct does admit a quasitrace (and likely also a trace). We characterize in different ways \Cs s that admit a faithful trace, respectively, where each quotient of the \Cs{} admits a faithful trace.
\end{abstract}

\section{Introduction}  The purpose of this article is to present a survey of a collection of important results on traces and quasitraces on a \Cs, and to add some new results to this collection. Throughout we have occasion to mention several profound results by Eberhard Kirchberg that offer interesting perspectives to the theory. The new results are in part contained in the PhD thesis of the first-named author.

In Section 2 we discuss the interplay between quasitraces and traces with emphasis on Kaplansky's question if all quasitraces on \Cs s are traces, and we review results on the existence of quasitraces on \Cs s. Section 3 is devoted to the existence of traces and elements of Haagerup's proof that all quasitraces on an exact \Cs{} are traces in the version of the Haagerup--Thorbj\o rnsen approach. Traces on ultraproducts of a sequence of \Cs s is the topic of Section 4, and we present a new example of  a sequence of unital simple separable nuclear traceless \Cs s whose ultraproduct admits a quasitrace. We leave it open if this quasitrace indeed is a trace. Section 5 contains a discussion of so-called ``almost traces'' which necessarily are possessed by any sequence of \Cs s whose ultraproduct admits a trace. Finally, in Section 6, we present conditions ensuring that an (exact) \Cs{} admits a \emph{faithful} trace, respectively, that each quotient of the \Cs{} admits a faithful trace. 

We thank the referee for several useful comments and for pointing out a gap in a proof in the first draft of this paper.

\section{Traces and quasitraces on $C^*$-algebras}

We review here some well-known results on traces and quasitraces on \Cs s with emphasis on the existence of quasitraces and Kaplansky's question if all quasitraces are traces. For convenience we mostly restrict our attention to unital \Cs s. A \emph{trace} on a unital \Cs{} $\cA$ is a bounded linear functional $\tau$ on $\cA$ satisfying the trace condition: $\tau(x^*x) = \tau(xx^*)$, for all $x \in \cA$. If $\tau$, moreover, is a state, then $\tau$ is said to be a \emph{tracial state}. Whenever we henceforth talk about a trace we shall mean a tracial state, unless otherwise stated.

A \emph{quasitrace} $\tau$ on $\cA$ is a  function $\tau \colon \cA \to \C$ satisfying
\begin{itemize}
\item $\tau(x^*x) = \tau(xx^*) \ge 0$, for all $x \in \cA$,
\item $\tau$ is linear on commutative sub-\Cs s of $\cA$,
\item $\tau(a+ib) = \tau(a)+i\tau(b)$, for $a,b \in \cA_\mathrm{sa}$.
\end{itemize}
If, moreover, there exists a quasitrace $\tilde{\tau}$ on $\cA \otimes M_2(\C)$ such that $\widetilde{\tau}(a \otimes e_{11}) = \tau(a)$, for all $a \in \cA$, then $\tau$ is said to be a $2$-quasitrace.\footnote{The extension of $\tau$ to $\widetilde{\tau}$ is non-canonical, unless $\tau$ is a trace.}  It was shown by Blackadar and Handelman in \cite{BlaHan:quasitrace} that every $2$-quasitrace extends to a quasitrace on $\cA \otimes M_n(\C)$, for all $n \ge 3$. They also showed,  
\cite[II.1.6]{BlaHan:quasitrace}, that $2$-quasitraces automatically are norm continuous. 
 Kirchberg proved (unpublished), see \cite[p.\ 1099]{Kir:book}, using a contruction of Aarnes of a non-additive quasi-state on $C([0,1]^2)$, that there exists quasitraces that are not $2$-quasitraces. \emph{For ease of notation we shall here reserve the term quasitrace to mean a $2$-quasitrace that is normalized.}

We let $T(\cA)$ and $QT(\cA)$ denote the sets of all traces, respectively, quasitraces on $\cA$. It is known that both of these sets (if non-empty) are \emph{Choquet simplices}, \cite[3.1.8]{Sak:C*-W*}, \cite[Theorem II.4.4]{BlaHan:quasitrace}, and that $T(\cA)$ is a face in $QT(\cA)$, \cite[Theorem II.4.5]{BlaHan:quasitrace}.

Quasitraces were first encountered in AW$^*$-algebras developed by Kaplansky in 1951--52, \cite{Kap:AW*-1951} and \cite{Kap:AW*-1952}. A unital \Cs{} is an AW$^*$-algebra if each of its MASAs are generated by projections and if each orthogonal family of projections has a least upper bound. AW$^*$-algebras can be divided into types as in the type decomposition for von Neumann algebras. Every AW$^*$-algebra of type II$_1$ admits a quasitrace, and an AW$^*$-factor of type II$_1$ is a von Neumann algebra if and only if it admits a trace, i.e., its unique quasitrace is a trace. While it is known that not all AW$^*$-algebras are von Neumann algebras, it remains an open problem --- raised by Kaplansky --- if every AW$^*$-factor of type II$_1$ is a von Neumann algebra. This question is equivalent to the following:

\begin{question}[Kaplansky] \label{q:Kap} Is every quasitrace on a \Cs{} a trace?
\end{question}

\noindent Suppose we know that every AW$^*$-factor $\cM$ of type II$_1$ is a von Neumann algebra (and hence that the unique quasitrace on $\cM$ is a trace). Let $\tau$ be an extremal quasitrace on some unital \Cs{} $\cA$. By performing a suitable completion of $\cA$ with respect to $\tau$, see, e.g., \cite[Proposition 3.12]{Haa:quasitraces} for details, one obtains a unital $^*$-homomorphism from $\cA$ into an AW$^*$-factor $\cM$ of type II$_1$, so that $\tau$ extends to the (unique) quasitrace $\overline{\tau}$ on $\cM$. Hence $\tau$ is a trace if $\overline{\tau}$ is a trace. 

The best known partial answer to Kaplansky's question was given by U.\ Haagerup, \cite{Haa:quasitraces}, that appeared as a preprint in 1991, and was published in 2014.  

\begin{theorem}[Haagerup] \label{thm:H}
Any quasitrace on any unital exact \Cs{} is a trace.
\end{theorem}

\noindent Kirchberg, \cite{Kir:quasitraces}, extended Haagerup's theorem to cover also non-unital \Cs s.

It is well understood when a unital \Cs{} possesses a quasitrace. Recall that a unital \Cs{} $\cA$ is said to be \emph{properly infinite} if it contains two isometries with orthogonal range projections (or, equivalently, two mutually orthogonal projections both of which are Murray-von Neumann equivalent to $1_\cA$). We say that $\cA$ is \emph{stably properly infinite} if $M_n(\cA)$ is properly infinite for some $n \ge 1$.  Similarly, we say that $\cA$ is \emph{stably finite} if $M_n(\cA)$ is finite, for all $n \ge 1$, where finite means that no projection is infinite, i.e., not equivalent to a proper subprojection of itself. 

\begin{theorem}[Cuntz, Blackadar--Handelman, \cite{Cuntz:dimension}, \cite{Handel:AW*}, \cite{BlaHan:quasitrace}] \label{thm:C-B-H}
A unital \Cs{}  admits a quasitrace if and only if it is not stably properly infinite.
\end{theorem}

\noindent In the literature, existence of quasitraces is often expressed by the sufficient (but not necessary) condition that the \Cs{} be stably finite, which is a stronger condition than \emph{not stably properly infinite}. For \emph{simple} \Cs s, the two properties ``stably finite'' and ``not stably properly infinite'', agree by Cuntz' observation that infinite projections in a simple \Cs{} automatically are properly infinite.

It is easy to see that no properly infinite \Cs{} can admit a quasitrace, and as quasitraces extend to matrix algebras by convention, the ``only if'' part of Theorem~\ref{thm:C-B-H} follows. 
To see the ``if'' part, consider the Cuntz semigroup $\mathrm{Cu}(\cA)$ of $\cA$ and the element $\langle 1 \rangle \in\mathrm{Cu}(\cA)$ representing the unit of $\cA$. Observe that if $\cA$ is not stably properly infinite,  the map $\N_0 \langle 1 \rangle \to \R$, given by $n\langle 1 \rangle \mapsto n$, is positive.  By \cite{BlaRor:extending}, this positive map extends to a state on $\mathrm{Cu}(\cA)$, i.e., a dimension function on $\cA$, normalized at the unit of $\cA$. A standard trick, cf.\  \cite[I.5]{BlaHan:quasitrace}, produces a lower semi-continuous  dimension function from this dimension function, still normalized at the unit of $\cA$. Finally, by 
\cite[Theorem II.2.2]{BlaHan:quasitrace}, any lower semi-continuous dimension function extends to a quasitrace.

Kirchberg proved the theorem below in the mid 1990s, cf.\ \cite[Theorem E]{Kir:book}, see also \cite[Theorem 4.1.10]{Ror:encyc}. At that time  it was not even known that the minimal tensor product of two simple purely infinite \Cs s is again purely infinite!

\begin{theorem}[Kirchberg] \label{thm:Kir-simple}
The minimal tensor product of two unital infinite-dimensional simple \Cs s is purely infinite if one of them is not stably finite.
\end{theorem}

\noindent The theorem leaves open the following:

\begin{question} \label{q:sxs}
Is the minimal tensor product of two unital simple stably finite \Cs s again stably finite?
\end{question}

\noindent If both \Cs s in the question above admit a tracial state, then so does the tensor product, which will entail that the tensor product is stably finite. By Theorems~\ref{thm:C-B-H} and \ref{thm:H},
this is the case if both \Cs s are exact.

As it turns out, Question~\ref{q:Kap}, Question~\ref{q:sxs} and the question if all AW$^*$ factors of type II$_1$ are von Neumann algebras are equivalent. We mentioned already the equivalence of the former with the latter. If Kaplansky's question has an affirmative answer, then all unital stably finite \Cs s admit a tracial state, which entails stable finiteness of the tensor product. Conversely, if Kaplansky's question has a negative answer, then there is an  AW$^*$-factor $\cA$ of type II$_1$ that is not a von Neumann algebra, hence has no trace. By (Haagerup's) Theorem \ref{thm:H-P} below, this would imply that $\cA \otimes C^*_\lambda(\mathbb{F}_\infty)$ is properly infinite, thus yielding a negative answer to Question~\ref{q:sxs}. 
Throughout this text ``$\otimes$'' denotes the minimal tensor product.

A unital simple \Cs{} $\cA$ is said to be \emph{tensorially prime} if it is not isomorphic to a tensor product $\cB \otimes \mathcal{C}$, with $\cB$ and $\mathcal{C}$ infinite dimensional.

\begin{corollary} Let $\cA$ be an  exact simple unital \Cs{} which is not tensorially prime. Then $\cA$ either admits a tracial state and is stably finite, or it is purely infinite. 
\end{corollary}

\noindent It is known that this dichotomy does not hold without assuming non-primeness, \cite{Ror:simple}.  The corollary holds without the assumption of $\cA$ being exact if and only if Kaplansky's question above has an affirmative answer, cf.\ the discussion above.

\section{Existence of traces on \Cs s}

\noindent We saw in the previous section that a unital \Cs{} admits a quasitrace if and only if it is not stably properly infinite. As we do not know if quasitraces are traces, the condition for having a trace (in the non-exact case) is more involved. We have the following powerful theorem by Haagerup, \cite{Haa:quasitraces}, with the extra condition (iv) added on by Pop, \cite{Pop:traces}:

\begin{theorem}[Haagerup, Pop] \label{thm:H-P} 
The following conditions are equivalent for each unital \Cs{} $\cA$:
\begin{enumerate}
\item $\cA$ admits no tracial state,
\item there exist $n \ge 2$ and $x_1, \dots, x_n \in \cA$ satisfying
$$\sum_{j=1}^n x_j^*x_j = 1_\cA, \qquad \big\| \sum_{j=1}^n x_jx_j^* \big\| < 1,$$
\item for each $\delta > 0$, there exist $n \ge 2$ and $x_1, \dots, x_n \in \cA$ satisfying
$$\sum_{j=1}^n x_j^*x_j = 1_\cA, \qquad \big\| \sum_{j=1}^n x_jx_j^* \big\| \le \delta,$$
\item there exists $n \ge 2$ such that each $a \in \cA$ is a sum of $n$ commutators from $\cA$,
\item $\cA \otimes C^*_\lambda(\mathbb{F}_\infty)$ is properly infinite.
\end{enumerate}
\end{theorem}

\noindent In (v), $C^*_\lambda(\mathbb{F}_\infty)$ denotes the reduced group \Cs{} associated with the free group, $\mathbb{F}_\infty$, with infinitely many generators ($\lambda$ refers to the left-regular representation). 

A few comments about the proof of this theorem: First, it is clear that each of the conditions (ii)--(v) imply (i). The implications (i) $\Rightarrow$ (iii) is obtained by passing to the bidual $\cA^{**}$, which is a properly infinite von Neumann algebra when (i) holds. Any properly infinite von Neumann algebra (or \Cs{}) contains a sequence $\{s_k\}_{k \ge 1}$ of isometries with pairwise orthogonal range projections. The elements $x_j = n^{-1/2} s_j$, $1 \le j \le n$, satisfy (iii) with $\delta = 1/n$. We review a proof of (iii) $\Rightarrow$ (v) below. C.\ Pop gave a beautiful self-contained short proof of (i) $\Rightarrow$ (iv) in \cite{Pop:traces}.

\begin{corollary} \label{cor:A} A unital \Cs{} $\cA$ admits a tracial state if and only if there exists a state $\rho$ on $\cA$ and $\lambda \ge 1$ such that $\rho(x^*x) \le \lambda \rho(xx^*)$, for all $x \in \cA$.
\end{corollary}

\noindent
The standard way to prove the existence of a tracial state in a finite von Neumann algebra is to first prove 
Corollary~\ref{cor:A} (often in a quite different way than the one presented here), and then prove the existence of states satisfying the condition of the corollary. See, e.g., \cite[Section 8.2]{KadRin:opII}.

\begin{proof} ``Only if'' is clear. Suppose that $\cA$ admits no tracial state. Let $\rho$ be any state on $\cA$. Let $\delta >0$, and let  $x_1, \dots, x_n \in \cA$  be as in Theorem~\ref{thm:H-P}~(iii). Then $\sum_{j=1}^n \rho(x_j^*x_j) = 1$ and $\sum_{j=1}^n \rho(x_jx_j^*) \le \delta$. It follows that $\rho(x_j^*x_j) \ge \delta^{-1} \rho(x_jx_j^*)$, for at least one $j$. As $\delta >0$ was arbitrary, this proves the ``if'' part of the corollary.
\end{proof}

\noindent One can rephrase Corollary~\ref{cor:A} as follows:

\begin{corollary} \label{cor:B} A unital \Cs{} $\cA$ admits no tracial state if and only if for each state $\rho$ on $\cA$ and each $\delta >0$ there exists $x \in \cA$ such that $\rho(x^*x) =1$ and $\rho(xx^*) \le \delta$.
\end{corollary}

\noindent One can probably not control the norm of the element $x \in \cA$ from Corollary ~\ref{cor:A} or ~\ref{cor:B}, see a further discussion on this matter in Remark~\ref{rem:traceless}. However, if $\cA$ is properly infinite, and hence admits a unital embedding of $\cO_\infty$, one can choose $x$ to be an isometry.

The following (standard) lemma is a quantitative version of the well-known fact that the Cuntz--Toeplitz algebra 
$\mathcal{T}_2$, generated by two isometries with orthogonal range projections, is semiprojective:

\begin{lemma} \label{lm:apprpi}
A unital \Cs{} $\cA$ is properly infinite if (and only if) it contains elements $a_1,a_2$ satisfying $\|a_j^*a_i - \delta_{ji} \, 1 \| < 1/2$.  
\end{lemma}

\begin{proof} Set $s_j = a_j|a_j|^{-1}$, $j=1,2$, noting that $a_j^*a_j$ is invertible with spectrum inside $[1/2,3/2]$. Then $s_1,s_2$ are isometries with range projections $p_j = s_js_j^*$. Since
$$\|p_1p_2\| = \|s_1^*s_2\| = \||a_1|^{-1}a_1^*a_2 |a_2|^{-1}\| \le \|a_1^*a_2\| \cdot \||a_1|^{-1}\|  \cdot \||a_2|^{-1}\| <1,$$
it follows that $up_1u^* \le 1-p_2$, for some unitary $u \in \cA$. Thus $us_1, s_2$ are isometries in $\cA$ with orthogonal range projections.
\end{proof}

\noindent The constant $1/2$ in the lemma above is best possible, as witnessed by taking $a_1=a_2 = \frac{1}{\sqrt{2}}$.

\begin{lemma}[Haagerup, \cite{Haa:quasitraces}] \label{lm:H2}
Let $\cA$ be a unital \Cs, and let $x_1, \dots, x_n \in \cA$ be such that $\sum_{j=1}^n x_j^*x_j = 1$. Let $\{s_j\}_{j \ge 1}$ be the canonical generators of the Cuntz algebra $\cO_\infty$, and set
$$t_1 = \sum_{j=1}^n x_j \otimes (s_j+s_j^*), \qquad t_2 =  \sum_{j=1}^n x_j \otimes (s_{j+n} + s_{j+n}^*),$$
in $\cA \otimes \cO_\infty$. Then  $\|t_j^*t_i - \delta_{ji} \, 1\| \le 2 \sqrt{\alpha} + \alpha$, where $\alpha = \| \sum_{j=1}^n x_jx_j^*\|$.  In particular, $C^*(1,t_1,t_2)$ is properly infinite if $2 \sqrt{\alpha} + \alpha < 1/2$.
 \end{lemma}
 
 \begin{proof} Set
 $$u_1 =  \sum_{j=1}^n x_j \otimes s_j, \quad u_2 =  \sum_{j=1}^n x_j \otimes s_{j+n}, \quad v_1 = \sum_{j=1}^n x_j \otimes s_j^*, \quad v_2 =  \sum_{j=1}^n x_j \otimes s_{j+n}^*.$$
A straightforward calculation shows that $u_k^*u_\ell = \delta_{k,\ell} \, 1_\cA \otimes 1_{\cO_\infty}$, and $v_kv_k^* = \sum_{j=1}^n x_jx_j^* \otimes 1_{\cO_\infty}$, so $\|v_k\| \le \sqrt{\alpha}$.  
 \end{proof}
 
 \noindent Using free probability theory, Haagerup proved the following lemma (included in the proof of \cite[Theorem 2.4]{Haa:quasitraces}):
 
 \begin{lemma}[Haagerup] \label{lm:H3} 
 There is an embedding $\varphi \colon C^*_{\lambda}(\mathbb{F}_\infty) \to \cO_\infty$ and a continuous func\-tion $f \colon \T \to \R$ such that $\varphi(f(u_j)) = s_j+s_j^*$, where $\{u_j\}_{j\ge 1}$ and $\{s_j\}_{j \ge 1}$ are the canonical generators of $C^*_{\lambda}(\mathbb{F}_\infty)$, respectively, $\cO_\infty$. 
 \end{lemma}
 
 \noindent Combining the three previous lemmas one obtains a proof of (ii) $\Rightarrow$ (v) of Theorem~\ref{thm:H-P}.
  
Following Haagerup--Thorbj\o rnsen, \cite{HaaTho:Annals}, we sketch how one can obtain a proof of Theorem~\ref{thm:H} using Theorem~\ref{thm:H-P} and the deep fact, obtained in  \cite{HaaTho:Annals}, that $C^*_\lambda(\mathbb{F}_\infty)$ is MF. One dosn't quite get the entire Theorem~\ref{thm:H}, but rather that any exact unital  \Cs{} which is not stably properly infinite admits a tracial state. To compactify the notation below, we write $M_k$ for the matrix algebra $M_k(\C)$.
 
 That $C^*_\lambda(\mathbb{F}_\infty)$ is MF means that there is a unital embedding $\Phi$:
 $$
\xymatrix{ & \prod_{m \ge 1} M_{k_m} \ar[d]^-{\pi}\\
C^*_\lambda(\mathbb{F}_\infty) \ar[r]^-\Phi & \prod_{m\ge 1} M_{k_m}/\bigoplus_{m \ge 1} M_{k_m}}
$$
Let $\bar{a}_j = \{a_j(m)\}_{m \ge 1} \in \prod_{m  \ge 1} M_{k_m}$ be self-adjoint lifts of $a_j:=\Phi(f(u_j))$, $1 \le j \le n$, with $\|\bar{a}_j\|=\|a_j\|$, and where $f$ is as in Lemma~\ref{lm:H3}. 

Let $\cA$ be a unital exact \Cs{} with no tracial state. We show that $\cA$ is stably properly infinite. Take $x_1, \dots, x_n \in \cA$ such that $\sum_{j=1}^n x_j^*x_j=1$ and $\alpha := \|\sum_{j=1}^n x_jx_j^*\| \le 1/25$.
Consider the elements 
\begin{alignat*}{4}
r_1 &= \sum_{j=1}^n x_j \otimes a_j,& \quad  r_2 & =\sum_{j=1}^n x_j \otimes a_{j+n}&  \quad \text{in} \quad  & \cA \otimes \big(\prod_{m \ge 1} M_{k_m}/\bigoplus_{m \ge 1} M_{k_m}\big)&,\\
r_1(m) & = \sum_{j=1}^n x_j \otimes a_j(m),& \quad  r_2(m) & = \sum_{j=1}^n x_j \otimes a_{j+n}(m)&  \quad \text{in} \quad  & \cA \otimes M_{k_m}, \quad m \ge 1.&
\end{alignat*}
Then $\|r_j^*r_i - \delta_{ji} 1\| \le 2\sqrt{\alpha} + \alpha < 1/2$ by Lemmas~\ref{lm:H2} and \ref{lm:H3}.
By the assumed exactness of $\cA$, the sequence
$$\xymatrix{0 \ar[r] & {\displaystyle{\bigoplus_{m \ge 1}}} \big(\cA \otimes M_{k_m} \big) \ar[r] & 
\cA \otimes {\displaystyle{\prod_{m \ge 1}}} M_{k_m} \ar[r]^-{\mathrm{id} \otimes \pi} &  
\cA \otimes  \Big({\displaystyle{\prod_{m \ge 1}}} M_{k_m}/ {\displaystyle{\bigoplus_{m \ge 1}}} M_{k_m}\Big)
\ar[r] & 0,}$$
is exact. For $x \in \cA \otimes \big(\prod_{m \ge 1} M_{k_m}\big) \subset\prod_{m \ge 1} \big( \cA \otimes M_{k_m})$, 
we can compute $\|x\| = \sup_{m \ge 1} \|(\mathrm{id} \otimes \pi_m)(x)\|$, where $\pi_m \colon \prod_{\ell \ge 1} M_{k_\ell} \to M_{k_m}$ is the quotient mapping. Exactness of the sequence above further implies that
$$\|(\mathrm{id} \otimes \pi)(x)\| = \limsup_{m \to \infty} \|(\mathrm{id} \otimes \pi_m)(x)\|.$$
This shows that
\begin{equation} \tag{$*$}
\limsup_{m \to \infty} \|r_j(m)^*r_i(m) - \delta_{ji} 1\| = \|r_j^*r_i - \delta_{ji} 1\| < 1/2.
\end{equation}
Hence $\|r_j(m)^*r_i(m) - \delta_{ji} 1\|  < 1/2$, $i,j=1,2$, for some $m \ge 1$, which by Lemma~\ref{lm:apprpi} implies that $\cA \otimes M_{k_m}$ is properly infinite.

One could speculate that there might be some hidden algebraic relations behind the elements $\{a_j(m)\}$, $1 \le j \le 2n$, $m \ge 1$, 
responsible for ($*$) to hold, when $r_j(m)$ are defined as above for any $x_1, \dots, x_n \in \cA$ satisfying $\sum_{j=1}^n x_j^*x_j=1$ and $\|\sum_{j=1}^n x_jx_j^*\| \le 1/25$. If so, then ($*$) would hold without assuming exactness of $\cA$, thus 
(almost) confirming Kaplansky's conjecture.  However, following the arguments of \cite[Proposition 4.9]{HaaTho:traces}, one can show that ($*$) in fact fails for suitable $x_1, \dots, x_n$ in the non-exact \Cs{} $\cA=C^*(\mathbb{F}_\infty)$, no matter what lift $\bar{a}_j$ we choose for $a_j$.

\section{Ultraproducts and traces} \label{sec:ultraproduct}

\noindent
Let $\{\cA_n\}_{n\ge 1}$ be a sequence of unital \Cs s and let $\omega$ be a free ultrafilter on $\N$.  Consider the ultraproduct
$$\prod_\omega \cA_n := \prod_{n \ge 1} \cA_n / \cI_\omega(\{\cA_n\}),$$
where 
$$\cI_\omega(\{\cA_n\}) =\big\{\{x_n\} \in \prod_{n\ge 1} \cA_n : \lim_\omega \|x_n\| = 0\big\} \; \lhd  \: \prod_{n\ge 1} \cA_n .$$
Denote the quotient mapping $\prod_{n \ge 1} \cA_n \to \prod_\omega \cA_n$ by $\pi_\omega$.

Let $T_\omega(\{\cA_n\}) \subseteq T(\prod_\omega \cA_n)$ denote the set of tracial states $\tau$ on $\prod_\omega \cA_n$ which are limits of a sequence $\{\tau_n\}$ of traces $\tau_n \in T(\cA_n)$, in the following sense:
$$\tau(\pi_\omega(\{x_n\})) = \lim_\omega \tau_n(x_n), \qquad \{x_n\} \in \prod_{n \ge 1} \cA_n.$$
The trace $\tau$ defined above will be denoted $\lim_\omega \tau_n$. It is well-defined since the limit on the right-hand side does not depend on the choice of the lift $\{x_n\}$. 

\begin{theorem}[Ozawa, \cite{Ozawa:Dixmier}] 
If $\{\cA_n\}$ is a sequence of unital exact $\cZ$-absorbing \Cs s, and if $\omega$ is a free ultrafilter on $\N$, then $T_\omega(\{\cA_n\})$ is weak$^*$ dense in  $T(\prod_\omega \cA_n)$.
\end{theorem}

\noindent Ozawa also remarked that the conclusion of his theorem fails without the assumption on $\cZ$-stability. More recently,  Antoine--Perera--Robert--Thiel, \cite{APRT:ultrapower}, gave several equivalent conditions for weak$^*$ density of  $T_\omega(\{\cA_n\})$  in  $T(\prod_\omega \cA_n)$ in general, and in particular for the case where $\cA_n=\cA$ is fixed. 

We proceed to construct a sequence $\{\cA_n\}$ of traceless simple, unital, separable, nuclear \Cs s where $\prod_\omega \cA_n$ admits a quasitrace. Recall that a \Cs{} is said to have property (SP) if each non-zero hereditary sub-\Cs{} contains a non-zero projection. 

\begin{lemma} \label{lm:n} Let $\cA$ be a unital  \Cs{} with property (SP) which admits no finite dimensional representations. Then, for all $n\ge 1$, there is a (possibly non-unital) embedding $M_n(\C) \to \cA$.
\end{lemma}

\begin{proof} By Glimm's lemma there is a non-zero \sh{} $\varphi \colon C_0((0,1]) \otimes M_n(\C) \to \cA$. (For this we only need that $\cA$ admits an irreducible representation of dimension at least $n$.)  Let $\iota \in C_0((0,1])$ denote the function $\iota(t) = t$, and let $\{e_{ij}\}\subseteq M_n(\C)$ be a set of matrix units. Set $a_j = \varphi(\iota \otimes e_{jj})$. Choose a non-zero projection $q_1$ in the hereditary sub-\Cs{} $\overline{a_1\cA a_1}$. Set $z_j=\varphi(\iota \otimes e_{j1})$, $2 \le j \le n$, and note that $|z_j| = a_1$. Write $z_j = v_j |z_j|$, with $v_j$ a partial isometry in $\cA^{**}$. Then $w_j:= v_jq_1$ is a partial isometry in $\cA$ satisfying $w_j^*w_j = q_1$ and $q_j:= w_jw_j^* \in \overline{a_j\cA a_j}$. It follows that $q_1, q_2, \dots, q_n$ are pairwise equivalent and pairwise orthogonal projections in $\cA$, witnessing the existence of the  embedding $M_n(\C) \to \cA$.
\end{proof}

\noindent The result below was obtained by the second named author in \cite{Ror:simple} and \cite{Ror:SP}:

\begin{theorem} \label{thm:P}
There exists a unital simple separable nuclear \Cs{} $\cP$ which is finite, but where $M_2(\cP)$ is properly infinite. Moreover, $\cP$ has property (SP), but is not of real rank zero.
\end{theorem}

\noindent We can sharpen the first part of the theorem above as follows:

\begin{theorem} \label{thm:simple} For each $n \ge 1$, there exists a unital \Cs{} $\cP_n$, Morita equivalent to the \Cs{} $\cP$ in the theorem above, such that $M_k(\cP_n)$ is finite, for $k \le n$, and properly infinite, for $k > n$.
\end{theorem}

\begin{proof}  Let $n \ge 1$ be given, and let $\cP$ be as in Theorem~\ref{thm:P} above. Use Lemma~\ref{lm:n} to find an embedding $\varphi \colon M_n(\C) \to \cP$. Let $e = \varphi(e_{11})$, where $\{e_{ij}\}$ are matrix units for $M_n(\C)$, and set $\cR = e\cP e$. Then $M_n(\cR) = \varphi(1) \cP \varphi(1)$, which is finite, because $\cP$ is finite. As $\cR$ and $\cP$ are Morita equivalent and since $M_2(\cP)$ is properly infinite, some matrix algebra over $\cR$ is properly infinite. Let $m \ge n$ be the largest integer for which $M_m(\cR)$ is finite. Write $m = qn + r$, with $0 \le r < n$, and put $\cP_n = M_q(\cR)$. Then $M_n(\cP_n) = M_{qn}(\cR)$ is finite, while $M_{n+1}(\cR) = M_{q(n+1)}(\cR)$ is properly infinite.
\end{proof}

\begin{lemma} \label{lm:omega}
Let $\{\cA_n\}$ be a sequence of unital \Cs s, let $\omega$ be a free ultrafilter on $\N$, and let $k \ge 1$. Then $M_k(\prod_\omega \cA_n)$ is properly infinite if and only if the set $\{n \ge 1 : M_k(\cA_n) \; \text{properly infinite}\}$ belongs to $\omega$.
\end{lemma}

\begin{proof} Since $M_k(\prod_\omega \cA_n) \cong \prod_\omega M_k(\cA_n)$, it suffices to prove the lemma for $k=1$. It is clear that $\prod_\omega \cA_n$ is properly infinite if $\cA_n$ is properly infinite for all $n$ in a set belonging to $\omega$. Conversely, suppose that $\prod_\omega \cA_n$ is properly infinite, and take isometries $s,t \in \prod_\omega \cA_n$ with orthogonal ranges. Write $s = \pi_\omega(\{s_n\})$ and $t = \pi_\omega(\{t_n\})$, with $s_n,t_n \in \cA_n$. Then
$$\lim_{n \to \omega} \|s_n^*s_n -1\| = \lim_{n \to \omega} \|t_n^*t_n -1\|  = \lim_{n \to \omega} \|t_n^*s_n \|  = 0.$$
In particular, each of $\|s_n^*s_n -1\|,  \|t_n^*t_n -1\|, \|t_n^*s_n \|$ is less than $1/2$, for all $n$ in some subset of $\omega$, and $\cA_n$ is properly infinite for those $n$, by Lemma~\ref{lm:apprpi}.
\end{proof}

\noindent The following is an immediate consequence of Lemma~\ref{lm:omega} above:

\begin{theorem} \label{thm:exotic-traces}
Let $\cP_n$ be as in Theorem~\ref{thm:simple}. Then each $\cP_n$ is traceless,  while $\prod_\omega \cP_n$ does admits a quasitrace. 
\end{theorem}

\noindent
 We know that the ultraproduct $\prod_\omega \cP_n$ from the theorem above admits a quasitrace, but we do not know if it admits a trace. It seems likely that it does (see also Remark~\ref{rem:Kap} below), in which case we will have an example where $T_\omega(\{\cP_n\}) = \emptyset$ and $T(\prod_\omega \cP_n) \ne \emptyset$. 
 
On the other hand, if $T(\prod_\omega \cP_n) = \emptyset$, we would have a counterexample to Kaplansky's conjecture: a 2-quasitrace which is not a trace. In either case we have an example of a quasitrace which is not obviously a trace.

In \cite[Example 2.12]{APRT:ultrapower} and \cite[Example 3.11]{Robert:Lie} an example was given of a sequence of (non-simple) unital traceless \Cs s $\cB_n$ whose ultraproduct $\prod_\omega \cB_n$ admits a character, and hence a trace. A related example was established in \cite{RobRor:divisible}, proving the existence  of a sequence of simple infinite dimensional unital \Cs s (admitting tracial states) whose ultraproduct also admits a character. 
We do not know if our example in Theorem~\ref{thm:exotic-traces} --- as it stands --- also would admit a character (we guess not), but it is easy to see that this is not the case for a suitable choice of the sequence $\{\cP_n\}$. Indeed, using Theorem~\ref{thm:simple}, we can for each $n \ge 1$ choose a unital \Cs{} $\mathcal{Q}_n$ Morita equivalent to $\cP$ such that $M_k(\mathcal{Q}_n)$ is finite for $k \le 2^n n$, and properly infinite for $k > 2^n n$. Put $\cP_n = M_{2^n}(\mathcal{Q}_n)$. Then $\cP_n$ satisfies the condition of Theorem~\ref{thm:simple}, and so $\prod_\omega \cP_n$ admits a quasitrace, and it also contains the UHF-algebra of type $2^\infty$ as a unital sub-\Cs.

\begin{remark} \label{rem:Kap}
We shall here comment on the possibility of constructing a counterexample to Kaplansky's conjecture using ultraproducts as in Theorem~\ref{thm:exotic-traces} above. For each unital \Cs{} $\cA$, let $\mu(\cA)$ be the smallest integer $n \ge 1$, for which $M_n(\cA)$ is properly infinite, and set $\mu(\cA)=\infty$, if no such $n$ exists. Similarly, let $\nu(\cA)$ be the smallest integer $n \ge 1$ for which there exist $x_1, \dots, x_n \in \cA$ satisfying $\sum_{j=1}^n x_j^*x_j=1$ and $\|\sum_{j=1}^n x_jx_j^*\| \le 1/2$, and set $\nu(\cA)=\infty$ if no such finite set $x_1, \dots, x_n$ exists. Then $\cA$ admits a trace, respectively, a quasitrace if and only if $\nu(\cA)=\infty$, respectively, 
$\mu(\cA)=\infty$, by Theorems~\ref{thm:H-P} and \ref{thm:C-B-H}.

If $\{\cA_n\}_{n \ge 1}$ is a sequence of unital \Cs s such that $\mu(\cA_n)  \to \infty$, then $\prod_\omega \cA_n$ is not stably properly infinite, and hence admits a quasitrace, for any ultrafilter $\omega$. If, at the same time, $\nu(\cA_n)$ remains bounded, then $\prod_\omega \cA_n$ cannot admit a trace, and such a  --- hypothetical --- \Cs{} will therefore provide a counterexample to Kaplansky's conjecture.  A construction of a sequence of \Cs s with these properties is impossible precisely if there exists
a universal relation between $\mu$ and $\nu$ of the form $\mu(\cA) \le f(\nu(\cA))$, for all traceless unital \Cs s $\cA$, where $f \colon \N \to \N$ is a fixed function satisyfing $f(n) \to \infty$ as $n \to \infty$. 
The results of this section provide a class of (simple) \Cs s $\cA$, exhausing all possible values of $\mu(\cA)$, 
which could be candidates for providing a counterexample to Kaplansky's conjecture via an ultraproduct construction.
\end{remark}

\section{Almost traces}

\noindent We introduce in this section a notion of \emph{almost traces}, and we show that an ultraproduct of a sequence of \Cs s admits a trace if and only if the \Cs s in the sequence admit almost traces. The arguments are very similar to those of \cite[Section 8]{RobRor:divisible} addressing existence of characters on ultraproducts of \Cs s.

As in the previous section, let $\{\cA_n\}$ be a sequence of unital \Cs s and let $\omega$ be a free ultrafilter on $\N$. Analogous with the construction for traces in the previous section, let $S_\omega(\{\cA_n\}) \subseteq S(\prod_\omega \cA_n)$ denote the set of states $\rho$ on $\prod_\omega \cA_n$ which are limits of a sequence $\{\rho_n\}$ of states $\rho_n \in S(\cA_n)$ in the following sense:
\begin{equation*} 
\rho(\pi_\omega(\{x_n\})) = \lim_\omega \rho_n(x_n), \qquad \{x_n\} \in \prod_\omega \cA_n.
\end{equation*}
We denote the state $\rho$ by $\lim_\omega \rho_n$. The result below is a slight modification of \cite[Lemma 2.5]{Kir:abel}
 by Kirchberg.
 
\begin{proposition}[Kirchberg] \label{prop:Kir}
The set $S_\omega(\{\cA_n\})$ is weak$^*$ dense in $S(\prod_\omega \cA_n)$. 
\end{proposition}

\noindent Observe that the statement above holds without any additional assumptions on the \Cs s $\cA_n$, unlike the analogous situation for traces discussed in Section~\ref{sec:ultraproduct}

\begin{proof} Note first that $S_\omega(\{\cA_n\})$ is convex. Indeed, if $\rho = \lim_\omega \rho_n$ and $\rho' = \lim_\omega \rho'_n$ belong to $S_\omega(\{\cA_n\}) $, then $t \rho + (1-t) \rho' = \lim_\omega (t \rho_n + (1-t) \rho'_n) \in S_\omega(\{\cA_n\})$, for $0 < t < 1$. Hence it suffices to show that each \emph{pure} state $\rho$ on $\prod_\omega \cA_n$ belongs to the weak$^*$ closure of $S_\omega(\{\cA_n\})$. 

Take a pure state $\rho$ on $\prod_\omega \cA_n$, let $\mathcal{F}$ be a finite subset of $\prod_\omega \cA_n$, and let $\ep >0$ be given. The pure state $\rho$ can be excised, \cite{AkeAndPed:excising}, so there exists a positive element $h \in \prod_\omega \cA_n$ satisfying $\|h\|=1$ and $\|h^{1/2}xh^{1/2}-\rho(x)h\|< \ep$, for all $x \in \mathcal{F}$. Write $h = \pi_\omega(\{h_n\})$, with $h_n \in \cA_n$ positive and $\|h_n\|=1$, and find states $\rho_n$ on $\cA_n$ such that $\rho_n(h_n) = \|h_n\|$. Put $\sigma = \lim_\omega \rho_n$. Then $\sigma(h) = \lim_\omega \rho_n(h_n) = 1$. It follows that 
$$|\sigma(x) - \rho(x)| =  \big|\sigma\big(h^{1/2}xh^{1/2} - \rho(x)h\big)\big| \le \| h^{1/2}xh^{1/2} - \rho(x)h\| < \ep,$$
for all $x \in \mathcal{F}$. This shows that $\rho$ belongs to the weak$^*$ closure of  $S_\omega(\{\cA_n\})$.
\end{proof}

\begin{remark} \label{rem:traceless}
In Corollary~\ref{cor:B} we noted that to each state $\rho$ on a unital traceless \Cs{} $\cA$ there exists a sequence $\{x_n\}$ in $\cA$ such that $\rho(x_n^*x_n)=1$, while $\rho(x_nx_n^*) \to 0$. One could consider the (formally) stronger property of a state that this happens for a \emph{bounded} sequence $\{x_n\}$. If all states on each \Cs{} in a sequence $\{\cA_n\}$ have this stronger property, then the same holds for all states in the weak$^*$ dense subset $S_\omega(\{\cA_n\})$, whence no state in $S_\omega(\{\cA_n\})$ could be a trace.  This would not rule out the existence of a trace on $\prod_\omega \cA_n$. Still, we expect that the sequences $\{x_n\}$ in $\cA$ above, witnessing tracelessness, cannot always be chosen to be bounded.
\end{remark}

\begin{definition} A unital \Cs{} $\cA$ is said to have \emph{$(N,\ep)$-almost traces} if for each set $x_1, \dots, x_N$ of contractions in $\cA$ there exists a state $\rho$ on $\cA$ such that
$$|\rho(x_j^*x_j - x_jx_j^*)| \le \ep, \qquad j=1,2,\dots, N.$$
\end{definition}

\begin{lemma} \label{lm:almost-trace}
A unital \Cs{} $\cA$ has a tracial state if and only if it has $(N,\ep)$-almost traces for all $N \ge 1$ and all $\ep >0$. 
\end{lemma}

\begin{proof} The ``only if'' part is clear. Assume that for each finite set $\mathcal{F}$ of the unit ball of $\cA$ and each $\ep >0$ we can find a state $\rho_{\mathcal{F},\ep}$ on $A$ such that $|\rho_{\mathcal{F},\ep}(x^*x-xx^*)| < \ep$ for all $x \in \mathcal{F}$. Let $\rho$ be an 
an accumulation point of the net $\big\{\rho_{\mathcal{F},\ep}\big\}$. Let $x \in \cA$ be a contraction and let $\ep > 0$. As $\rho$ by assumption belongs to the closure of the set $\{\rho_{\mathcal{F},\ep'} : x \in \mathcal{F}, \; \ep' \le \ep\}$ we conclude that $|\rho(x^*x-xx)|<\ep$. This shows that $\rho$ is a tracial state on $\cA$.
\end{proof}

\begin{proposition} \label{prop:UH}
Let $\{\cA_n\}$ be a sequence of unital \Cs s, and let $\omega$ be a free ultrafilter on $\N$. The ultraproduct \Cs{} $\prod_\omega \cA_n$ admits a tracial state if and only if 
$$I_{N,\ep}:=\{n \in \N : \cA_n \; \text{admits $(N,\ep)$-almost traces} \, \} \in \omega,$$
for each integer $N \ge 1$ and each $\ep >0$.
\end{proposition}

\begin{proof} 
``Only if''. Let $\tau$ be a tracial state on $\prod_\omega \cA_n$. Let $(N,\ep)$ be given. For each $n \in \N \setminus I_{N,\ep}$, choose contractions $x_1(n), \dots, x_N(n) \in \cA_n$ such that no state $\rho$ on $\cA_n$ satisfies $|\rho\big(x_j(n)^*x_j(n)-x_j(n)x_j(n)^*\big)|\le \ep$, for all $j=1,\dots, N$. Choose arbitrary contractions $x_j(n) \in \cA_n$, for $n \in I_{N,\ep}$ and $j=1,\dots,N$, and set $x_j = \pi_\omega(\{x_j(n)\}) \in \prod_\omega \cA_n$. 

By Proposition~\ref{prop:Kir} we can find states $\rho_n$ on $\cA_n$ such that 
$$|(\lim_\omega \rho_n)(x_j^*x_j) - \tau(x_j^*x_j)| < \ep/2, \qquad |(\lim_\omega \rho_n)(x_jx_j^*) - \tau(x_jx_j^*)| < \ep/2,$$ for all $j=1,\dots, N$. Hence 
$$\lim_\omega |\rho_n\big(x_j(n)^*x_j(n)-x_j(n)x_j(n)^*\big)| = |(\lim_\omega \rho_n)(x_j^*x_j-x_jx_j^*) | < \ep,$$ for all $j$. This entails that $\N \setminus I_{N,\ep}$ does not belong to $\omega$, so $I_{N,\ep} \in \omega$ as desired.

``If''. By Lemma~\ref{lm:almost-trace} it suffices to show that $\prod_\omega \cA_n$ admits $(N,\ep)$-almost traces for all $(N,\ep)$. Accordingly, let $x_1, \dots, x_n$ be contractions in $\prod_\omega \cA_n$, and write $x_j = \pi_\omega(\{x_j(n)\}_{n \ge 1})$ with $x_j(n) \in \cA_n$ contractions. For each $n \in I_{N,\ep}$, we can find a state $\rho_n$ on $\cA_n$ such that $|\rho_n(x_j(n)^*x_j(n)-x_j(n)x_j(n)^*)| \le \ep$, for $j=1,2, \dots, N$. Choose arbitrary states $\rho_n \in S(\cA_n)$, for $n \notin I_{N,\ep}$, and set $\rho = \lim_\omega \rho_n$. Then
$$|\rho(x_j^*x_j-x_jx_j^*)| = \lim_\omega |\rho_n(x_j(n)^*x_j(n)-x_j(n)x_j(n)^*)| \le \ep,$$
for $j=1,2, \dots, N$.
This shows that $\prod_\omega \cA_n$ admits $(N,\ep)$-almost traces.
\end{proof}

 \section{Faithful traces}
 
 \noindent In this last section we shall consider when a unital \Cs{} admits a \emph{faithful} tracial state, and also the stronger condition that each \emph{quotient} of the \Cs{} admits a tracial state.  The property that each quotient of the \Cs{} admits a tracial state is denoted QTS, and was considered by Murphy, \cite{Murphy:traces}, see also recent application for amenability of the unitary group of the \Cs{} in \cite{AST:amenable} and \cite{Ozawa:amenable}.

We begin by noting the following well-known characterization of \Cs s with a faithful trace. The equivalence of (i) and (ii) was obtained in  \cite[Theorem 3.4]{CuntzPed-JFA} (we give a different proof here). For each trace $\tau$ on a \Cs{} $\cA$ consider the its trace-kernel ideal $I_\tau =  \{ x \in \cA : \tau(x^*x) = 0\}$.

\begin{proposition} \label{prop:faithful}
The following three conditions are equivalent for every unital separable \Cs{} $\cA$:
\begin{enumerate}
\item $\cA$ has a faithful tracial state,
\item $\cA$ has a separating family of tracial states,
\item each non-zero closed two-sided ideal in $\cA$ admits a non-zero positive bounded trace.
\end{enumerate}
\end{proposition}

\begin{proof} (i) $\Rightarrow$ (ii) $\Rightarrow$ (iii) are trivial.

(ii) $\Rightarrow$ (i). We first show that for each non-zero positive element $a \in \cA$ there exists a tracial state $\tau$ on $\cA$ such that $\|a+I_\tau\| > \|a\|/2$, where $I_\tau$ is the trace-kernel ideal defined above. We may assume that $\|a\| = 1$. Let $g \colon [0,1] \to [0,1]$ be a continuous function which is zero on the interval $[0,1/2]$ and with $g(1) = 1$. Then $g(a)$ is positive and non-zero, so by assumption there is a tracial state $\tau$ on $\cA$ which is non-zero on $g(a)$. It follows that $g(a+I_\tau) = g(a) + I_\tau \ne 0$. This entails that $\|a+I_\tau\| > 1/2$.

Let now $\{a_n\}_{n=1}^\infty$ be a countable dense subset of the set of positive elements in $\cA$ of norm 1. For each $n$ choose a tracial state $\tau_n$ such that $\|a_n+ I_{\tau_n}\| > 1/2$. Set $\tau = \sum_{n=1}^\infty 2^{-n} \tau_n$. Then $\tau$ is a tracial state on $\cA$ and $I_\tau = \bigcap_{n=1}^\infty I_{\tau_n}$. It follows that $\|a_n+I_\tau\| > 1/2$ for all $n$. This implies that $\|a+I_\tau\| \ge 1/2$, for all positive elements $a \in \cA$ with $\|a\|=1$, proving that $\tau$ is faithful. 

(iii) $\Rightarrow$ (ii). Each trace on $\cA$ vanishes on the ideal $I_0 = \bigcap_{\tau \in T(\cA)} I_\tau$. Hence, $I_0$ does not admit any positive bounded trace, since any such would extend to a positive trace on $\cA$, so $I_0$ must be zero. 
\end{proof}

\noindent The conclusions of Proposition~\ref{prop:faithful} do not hold without the assumption of separability. If a \Cs{} admits a faithful tracial state, then it cannot contain an uncountable family of pairwise orthogonal non-zero positive elements, but such a \Cs{}  can admit a separating family of tracial states. This is for example the case for the \Cs{} $\ell^\infty(X)$, or the unitization of $c_0(X)$, where $X$ is an uncountable set. 

The following is an immediate corollary to Proposition~\ref{prop:faithful}:

\begin{corollary} Every separable RFD \Cs{} admits a faithful tracial state.
\end{corollary}

\noindent We mention also the following nice characterization from \cite[Theorem 3.4]{CuntzPed-JFA} of \Cs s admitting a separating family of traces:

\begin{theorem}[Cuntz--Pedersen]
A unital \Cs{} $\cA$ admits a separating family of traces if and only if for all sequences $\{x_n\}_{n \ge 1}$ in $ \cA$ satisfying $\sum_{j=1}^\infty x_jx_j^* \le \sum_{j=1}^\infty x_j^*x_j$ (with both sums being norm convergent) we have equality $\sum_{j=1}^\infty x_jx_j^* = \sum_{j=1}^\infty x_j^*x_j$.
\end{theorem}

\noindent We proceed to list some permanence properties of \Cs s with a faithful tracial state.

\begin{proposition} \label{prop:perm-fts}
Let $\cA$ and $\cB$ be  unital \Cs s. 
\begin{enumerate}
\item If $\cA$ and $\cB$ are Morita equivalent, and if $\cA$ admits a faithful trace, then so does $\cB$.
\item The minimal tensor product $\cA \otimes \cB$ admits a faithful trace if and only if both $\cA$ and $\cB$ do.
\item If $I$ is an essential ideal in $\cA$, and if $I$ admits a faithful bounded trace, then so does $\cA$. In particular, $I$ admits a faithful trace if and only if $\cM(I)$ does. 
\item If $I$ is an ideal in $\cA$, then $\cA$ admits a faithful trace if $I$ and $\cA/I$ do. 
\end{enumerate}
\end{proposition}

\begin{proof} (i). This follows from the easily seen facts that if $\cA$ admits a trace then so does any matrix algebra over $\cA$ and every corner $p\cA p$ of $\cA$, with $p$ a projection in $\cA$.

(ii). The ``only if'' is trivial because $\cA$ and $\cB$ both are sub-\Cs s of $\cA \otimes \cB$. To see the ``if'' part, take faithful traces $\tau_\cA$ and $\tau_\cB$ on $\cA$ and $\cB$, respectively, and consider the trace $\tau = \tau_\cA \otimes \tau_\cB$ on $\cA \otimes \cB$. To show that $\tau$ is faithful we must show that its associated trace-kernel ideal $I_\tau$ of $\cA \otimes \cB$ is zero. However, if $I_\tau$ were non-zero, then by Kirchberg's ``slice lemma'' (see, e.g., \cite[Lemma 4.1.9]{Ror:encyc}) it contains a non-zero elementary tensor $a \otimes b$. But this would contradict faithfulness of $\tau_\cA$ and $\tau_\cB$.

(iii). Take a faithful trace $\tau$ (of norm one) on $I$ and extend it to a trace $\overline{\tau}$ on $\cA$. Then $I_{\overline{\tau}} \cap I = I_\tau = 0$, which entails that $I_{\overline{\tau}} =0$, since $I$ is essential. 

(iv). Suppose that $I$ and $\cA/I$ both admit faithful traces, and  let $I_0 = \bigcap_{\tau \in T(\cA)} I_\tau$. 
Then $I \cap I_0 = 0$ because $I$ admits a faithful trace, and $(I+I_0)/I=0$ because $\cA/I$ admits a faithful trace, so  $I_0=0$.
\end{proof}

\noindent We emphasize below the following special case of part (iii) of the proposition above.

\begin{corollary} A  (possibly non-unital) \Cs{} $I$ admits a faithful positive bounded trace if and only if its multiplier algebra $\cM(I)$ and/or its unitization $\widetilde{I}$ admit a faithful trace.
\end{corollary}

\begin{remark}[Obstructions to having a faithful trace] \label{rem:obstructions}
Every unital \Cs{} with a faithful tracial state must be stably finite (all projections in all matrix algebras over the \Cs{} are finite). Indeed, any faithful trace on a \Cs{} extends to a faithful trace on any matrix algebras over the \Cs, and 
the presence of a faithful trace forces all projections to be finite.

Secondly, a faithful trace on a \Cs{} restricts to a non-zero bounded positive trace on all the its sub-\Cs s. Stable \Cs s do not admit any bounded non-zero trace, so no \Cs{} with a faithful trace can contain a stable sub-\Cs.\footnote{A \Cs{} $\cA$ is stable if $\cA \otimes \cK \cong \cA$, where $\cK$ is the \Cs{} of compact operators on an infinite dimensional separable Hilbert space, see also \cite{HjeRor:stable} for an intrinsic characterization of stability} 

Any \Cs{} that contains a purely infinite \Cs{} cannot have a faithful trace, as purely infinite \Cs s are traceless.  Purely infinite \Cs s need not have projections, nor have stable sub-\Cs s, so this obstruction is not covered by the two previously mentioned ones. 
\end{remark}

\noindent It follows in particular from Proposition~\ref{prop:perm-fts} that $\cA$ admits a faithful trace if and only if its $\cZ$-stabilization $\cA \otimes \cZ$ admits a faithful trace.

\begin{definition}
A \Cs{} $\cA$ will be said to be \emph{$n$-stable}, for some $n \ge 1$, if $M_n(\cA)$ is stable. We say that $\cA$ is 
\emph{matrix stable} if it is $n$-stable, for some $n \ge 1$.  
\end{definition}

\noindent
More generally, we can consider the property that $\cA \otimes \cB$ is stable, where $\cB$ is some fixed unital \Cs. For two very special cases of $\cB$ we can say the following:

\begin{theorem}  \label{prop:Z-stable} Let $\cA$ be a separable \Cs.
\begin{enumerate}
\item If $\cA$ is exact, then $\cA \otimes \cZ$ is stable if and only if $\cA$ admits no bounded trace and has no unital quotient.
\item $\cA \otimes \cO_\infty$ is stable if and only if $\cA$ has no unital quotient.
\end{enumerate}
\end{theorem}

\begin{proof} (i) follows from \cite[Theorem 3.6]{HirRorWin:fields} and \cite{Ror:Z}, while (ii) follows from \cite[Theorem 4.24]{KirRor:pi}, which states the more general result that any separable purely infinite \Cs{} is stable if and only if it has no unital quotient.
\end{proof}

\noindent One can replace $\cZ$ in (i) by any unital \Cs{} that tensorially absorbs $\cZ$.

\begin{remark} \label{rem:sns}
One can, for each $n \ge 2$, find a simple separable \Cs{} of stable rank one which is $n$-stable but not $(n-1)$-stable, cf.\ \cite{Ror:sns}. 

Taking a $c_0$-direct sum of a sequence of such \Cs s, one obtains a \Cs{} $\cA$ of stable rank one, which is not matrix stable, but where $\cA \otimes \cZ$ is stable. In the example below we construct a \emph{simple} (but not stable rank one) \Cs{} which likewise is stable after being tensored with $\cZ$, but not matrix stable. 
\end{remark}

\begin{example} \label{ex:matrixstable}
Take the \Cs{} $\cP$ from Theorem~\ref{thm:P}. Let $e_0$ denote the identity of $M_2(\cP)$, which is a properly infinite projection. Let $e_1$ be a non-zero finite projection in $M_2(\cP)$ such that $e_0-e_1$ is properly infinite. 
As in the proof of Theorem~\ref{thm:simple}, choose inductively non-zero projections $e_n \in M_2(\cP)$, for $n \ge 2$, such that $e_n \otimes 1_n$ is finite, and $e_n$ is a proper subprojection of $e_{n-1}$, where $e \otimes 1_n = \mathrm{diag}(e,e, \dots, e) \in M_n(M_2(\cP))$, for $e \in M_2(\cP)$. Put
$$\cA = \overline{\bigcup_{n=1}^\infty (e_0-e_n) M_2(\cP) (e_0-e_n)}.$$
Then $\cA$ is simple and non-unital (hence no unital quotients), and $\cA \otimes \cZ$ is purely infinite,  because $\cA$ is simple and not stably finite, so $\cA \otimes \cZ$ is stable by \cite[Theorem 4.24]{KirRor:pi}. Take $n \ge 1$, and observe that $\{(e_0-e_k) \otimes 1_n\}_{k \ge 1}$ is an approximate unit for $M_n(\cA)$. If $f \in M_n(\cA)$ is a projection which is orthogonal to $(e_0-e_n) \otimes 1_n$, then $f \precsim (e_n - e_k) \otimes 1_n \le e_n \otimes 1_n$, for some $k > n$. This implies that $f$ is finite, so no properly infinite projection is sub-equivalent to $f$. In particular, 
we cannot have $(e_0-e_n) \otimes 1_n \precsim f$. By \cite{HjeRor:stable} this shows that $M_n(\cA)$ is not stable. 
\end{example}

\noindent
The property of having a faithful trace (or a separating family of traces) is not closed under \emph{maximal} tensor product. In fact, we have the following remarkable result of Kirchberg from \cite{Kir:WEP}, see
his proof of (B4) $\Rightarrow$ (B3) on p.\ 485, see also \cite[Proposition 3.13]{KirRor:IJM} for a more detailed proof.

\begin{proposition}[Kirchberg] \label{prop:K-berg}
Let $\cA$ and $\cB$ be unital \Cs s. Then 
$\cA \otimes_{\mathrm{max}} \cB$ admits a faithful trace if and only if $\cA \otimes_{\mathrm{max}} \cB = \cA \otimes \cB$ and both $\cA$ and $\cB$ admit faithful traces.
\end{proposition}

\begin{remark} Kirchberg proved in \cite{Kir:WEP} that the Connes Embedding Problem (CEP) has a positive answer if and only if $C^*(\mathbb{F}_\infty) \otimes_{\mathrm{max}} C^*(\mathbb{F}_\infty) = C^*(\mathbb{F}_\infty) \otimes  C^*(\mathbb{F}_\infty)$. Since $C^*(\mathbb{F}_\infty)$ has a faithful tracial state, being RFD, Kirchberg thus shows that CEP is equivalent to the existence of a faithful trace on $C^*(\mathbb{F}_\infty \times \mathbb{F}_\infty) \cong C^*(\mathbb{F}_\infty) \otimes_{\mathrm{max}} C^*(\mathbb{F}_\infty)$. A negative solution to the CEP has been announced in \cite{JNVWY:MIP*=RE}, which in this context implies that $C^*(\mathbb{F}_\infty) \otimes_{\mathrm{max}} C^*(\mathbb{F}_\infty) \ne C^*(\mathbb{F}_\infty) \otimes  C^*(\mathbb{F}_\infty)$ and that $C^*(\mathbb{F}_\infty \times \mathbb{F}_\infty)$ has no faithful trace.

For the \emph{reduced} group \Cs s it is known, \cite{AkeOst:F_2-tensor}, that $C^*_\lambda(\mathbb{F}_\infty) \otimes_{\mathrm{max}} C^*_\lambda(\mathbb{F}_\infty) \ne C^*_\lambda(\mathbb{F}_\infty) \otimes C^*_\lambda(\mathbb{F}_\infty)$, so the former does not admit a faithful tracial state, while $C^*_\lambda(\mathbb{F}_\infty)$ surely does.

It is known that there exists groups $\Gamma$ for which $C^*(\Gamma)$ does not admit a faithful tracial state, e.g., $\Gamma = \mathrm{SL}(n,\Z)$, for $n \ge 3$, cf.\ \cite{Bekka-Invent-07}. We do not know which $C^*$-algebraic obstruction prevents these \Cs s having a faithful trace. One interesting possibility is that they contain an infinite projection. We mention here that it is an  open problem if there exists a group $\Gamma$ for which $C^*(\Gamma)$ contains an infinite projection. 
\end{remark}

\noindent Similarly one can ask if there is an alternative understanding of Proposition~\ref{prop:K-berg}:

\begin{question} Is the kernel of the canonical map $\cA \otimes_{\mathrm{max}} \cB \to \cA \otimes \cB$ always stable (or matrix stable), whenever $\cA$ and $\cB$ are unital \Cs s?
\end{question}

\begin{example} \label{ex:non-permanence}
(i). A quotient of a \Cs{} with a faithful trace need not admit a (faithful) trace. The full group \Cs{} $C^*(\mathbb{F}_\infty)$ is RFD and hence admits a faithful trace. However, every separable \Cs{} (including $\cO_\infty$) is a quotient of $C^*(\mathbb{F}_\infty)$.

(ii). The class of unital \Cs s admitting a faithful tracial state is not closed under inductive limits. The unitization $\widetilde{\cK}$ of the compact operators is an AF-algebra and hence an inductive limit of finite dimensional \Cs s, but it does not admit a faithful tracial state.

\end{example}

\begin{proposition} A unital \Cs{} has a faithful trace if and only if it embeds into a type II$_1$ von Neumann factor.
\end{proposition}

\begin{proof} If a unital \Cs{} $\cA$ admits a faithful tracial state $\tau$, then it embeds into the finite von Neumann algebra $\cM = \pi_\tau(\cA)''$, which again admits a faithful trace. It is well-known that every finite von Neumann algebra with a faithful tracial state embeds (in a trace preserving way) into a II$_1$-factor. For a proof of this fact, due to Haagerup, see the appendix of \cite{MusRor:CMP}.
\end{proof}

\noindent A unital \Cs{} embeds into a von Neumann algebra of type II$_1$ (not necessarily a factor) if and only if it admits a separating family of traces. 

In the theorem below we combine Theorem~\ref{prop:Z-stable} (i) and Proposition~\ref{prop:faithful} to get information about when a unital exact \Cs{} admits a faithful tracial state.

\begin{theorem} \label{thm:faithful} Let $\cA$ be a  unital exact separable \Cs.
\begin{enumerate}
\item If $\cA$ admits a faithful tracial state, then no non-zero ideal $I$ of $\cA$  is such that $I \otimes \cZ$ is stable.
\item If  $\cA$ has no non-zero ideal $I$ for which either $I \otimes \cZ$ is stable or $I$ has a stably properly infinite unital quotient, then $\cA$ admits a faithful tracial state.
\item If $\cA$ has stable rank one, then $\cA$ admits a faithful tracial state if and only if there is no non-zero ideal $I$ of $\cA$ for which $I \otimes \cZ$ is stable.
\end{enumerate}
\end{theorem}

\begin{proof} (i) is contained in Theorem~\ref{prop:Z-stable} (and is also trivial, as a faithful tracial state on $\cA$ restricts to a non-zero bounded trace on $I$ which, when tensored with the canonical trace on $\cZ$, gives a bounded trace on $I \otimes \cZ$, which therefore cannot be stable). 

(ii). By  Proposition~\ref{prop:faithful} it suffices to show, under the assumptions in (ii), that every non-zero ideal $I$ of $\cA$ admits a non-zero bounded trace. By Theorem~\ref{prop:Z-stable} (i) and the assumption that $I \otimes \cZ$ is not stable, either $I$ has a bounded trace or a unital quotient $I/J$ (or both). Assume that $I/J$ is unital for some $J \triangleleft I \triangleleft \cA$. Then, by assumption,  
$I/J$  is not stably properly infinite, so it admits a quasitrace $\tau$, by Theorems \ref{thm:C-B-H}. Now, since $\cA$ is exact, so is $I$ and  $I/J$, the latter is a deep fact obtained by Kirchberg,  \cite{Kir:UHF},  see also \cite[Corollary 9.3]{Was:exact}. It therefore follows from Haagerup's theorem (Theorem~\ref{thm:H}) that $\tau$ in fact is a trace on $I/J$. By composition, $\tau$ lifts to a bounded trace on $I$, as desired.

(iii). If $\cA$ has stable rank one, then so do all ideals of $I$ and all quotients of all ideals of $\cA$. As any \Cs{} of stable rank one is stably finite, we conclude that no ideal of $\cA$ has a stably properly infinite quotient.  The result now follows from (i) and (ii).
\end{proof}

\begin{remark} (i). If $\cA$ is $\cZ$-stable, then we can replace ``$I \otimes \cZ$ is stable'' by ``$I$ is stable'' in Theorem~\ref{thm:faithful}  since any ideal of a $\cZ$-stable \Cs{} automatically is $\cZ$-stable. In (ii) we can also replace ``$I$ has a stably properly infinite unital quotient'' by ``$I$ has a properly infinite unital quotient''.

(ii). We can not in general replace ``$I \otimes \cZ$ is stable'' by ``$I$ is stable'' in  Theorem~\ref{thm:faithful} (ii) or (iii). Indeed, for any integer $n \ge 2$, take a simple separable \Cs{} $I$ of stable rank one which is $n$-stable but not $(n-1)$-stable, cf.\ Remark~\ref{rem:sns}. Then $I$ admits no bounded trace, so its unitization $\widetilde{I}$ does not admit a faithful tracial state, but $\widetilde{I}$  has no stable ideal and no unital  quotient other than $\C$.

Also, one can not replace ``$I \otimes \cZ$ is stable'' by ``$I$ is matrix stable''  in  Theorem~\ref{thm:faithful} (ii), 
by Example~\ref{ex:matrixstable} and the argument above.

(iii). We do not know, however, if one can replace ``$I \otimes \cZ$ is stable'' by ``$I$ is matrix stable''   in Theorem~\ref{thm:faithful} (iii). Suppose that $\cA$ is a unital exact separable \Cs{} of stable rank one, and let $I$ be an ideal in $\cA$, which is not matrix stable.  Then, for each $n \ge 1$, it follows from \cite[Proposition 3.6]{Ror:sums} that $M_n(\cM(I)) = \cM(M_n(I))$ is not properly infinite, so $\cM(I)$ admits a quasitrace by Theorem~\ref{thm:C-B-H}. The restriction of such a quasitrace to $I$ will be a bounded trace, because $I$ is exact. However, it may be zero. Still, we do not know of an example of a \Cs{} $I$ of stable rank one,  for which $\cM(I)$ admits a quasitrace, but all quasitraces on $\cM(I)$ vanish on  $I$. Such examples do exists if we omit the requirement of having stable rank one, see (v) below. 

(iv). Relatedly, we do not know if one can find an example of a simple separable \Cs{} $\cA$ of stable rank one for which $\cA \otimes \cZ$ is stable but $\cA$ is not matrix stable, cf.\ Example~\ref{ex:matrixstable} and Remark~\ref{rem:sns}. If such an example exists, which additionally is exact, then its unitation would have no matrix stable ideals, but will   fail to have a faithful trace.

(v). Take $\{\cP_n\}_{n \ge 1}$ to be the sequence of unital simple traceless \Cs s from 
Theorem~\ref{thm:exotic-traces}. Then its $c_0$-direct sum $\bigoplus_{n=1}^\infty \cP_n$ admits no bounded trace (or quasitrace, as it is nuclear), but its multiplier algebra $\cM(\bigoplus_{n=1}^\infty \cP_n) = \prod_{n=1}^\infty \cP_n$ does admit a quasitrace, since a quotient of it does, cf.\ Theorem~\ref{thm:exotic-traces}. 
\end{remark}

\noindent
It would be interesting to have a necessary and sufficient condition for the existence of a faithful trace on a general (separable and exact) \Cs, finding a suitable ``middle road'' of (i) and (ii) in Theorem~\ref{thm:faithful} above. Let us here remark that the condition in (i) is not sufficient, and the conditions of (ii) are not necessary. Indeed, any unital properly infinite simple \Cs{} (e.g., $\cA = \cO_n$) satisfies (i) but has no trace. It was already noted in Example~\ref{ex:non-permanence} (i) that \Cs s with a faithful trace can have a properly infinite quotient.

The situation becomes more clear if, instead, we ask that $\cA$ and \emph{all quotients} of $\cA$ admit a faithful trace. We say that a \Cs{} has the QFTS property if all its quotients admit a faithful trace.
If $I \triangleleft J \triangleleft \cA$ are ideals, we refer to $J/I$ as an \emph{intermediate quotient} of $\cA$.

\begin{theorem} A unital exact separable \Cs{} $\cA$ has the QFTS property if and only if it has no stably properly infinite quotients and no intermediate quotient that is stable when tensored with $\cZ$. 

If $\cA$ moreover is $\cZ$-stable, then $\cA$ has the QFTS property if and only if it has no stable intermediate quotients and no properly infinite quotients. 
\end{theorem}

\begin{proof} We first prove the second part of the theorem. The argument resembles the proof of Theorem~\ref{thm:faithful} (ii). The ``only if'' part is easy and follows from Remark~\ref{rem:obstructions}. 

Suppose $\cA$ does not have the QFTS property, and let $I \triangleleft \cA$ be such that $\cA/I$ has no faithful trace. Then, by Proposition~\ref{prop:faithful}, there is $I \triangleleft J \triangleleft \cA$  such that $J/I$ has no bounded trace. Being $\cZ$-stable passes to intermediate ideals, so either $J/I$ is stable or has a unital quotient, by Theorem~\ref{prop:Z-stable}. In the latter case we will have an ideal $I \triangleleft K \triangleleft J \triangleleft \cA$ such that $J/K$ is unital. Since $J/I$ has no bounded trace, neither does $J/K$, and as $J/K$ is exact by Kirchberg's theorem that exactness passes to quotients, mentioned earlier, we infer that $J/K$ is stably properly infinite, and hence properly infinite, by $\cZ$-stability. Finally, since $J/K$ is a unital ideal in $\cA/K$, it is a direct summand, and therefore $J/K$ is isomorphic to a quotient of $\cA$.

The first part of the theorem follows from the second by applying it to $\cA \otimes \cZ$. Use that $\cA$ has the QTFS property if and only if $\cA \otimes \cZ$ does, and that $\cB \otimes \cZ$ is  properly infinite if and only if $\cB$ is stably properly infinite.
\end{proof}

{\small{
\bibliographystyle{amsplain}


\providecommand{\bysame}{\leavevmode\hbox to3em{\hrulefill}\thinspace}
\providecommand{\MR}{\relax\ifhmode\unskip\space\fi MR }
\providecommand{\MRhref}[2]{%
  \href{http://www.ams.org/mathscinet-getitem?mr=#1}{#2}
}
\providecommand{\href}[2]{#2}


%
%
}}

\vspace{1cm}

\noindent
\begin{tabular}{ll}
Henning Olai Milh\o j  & Mikael R\o rdam \\
Department of Mathematical Sciences & Department of Mathematical Sciences\\
University of Copenhagen & University of Copenhagen\\ 
Universitetsparken 5, DK-2100, Copenhagen \O & Universitetsparken 5, DK-2100, Copenhagen \O \\
Denmark & Denmark\\
 henningmilhoj@gmail.com & rordam@math.ku.dk
\end{tabular}


%
%

\end{document}